\documentclass[a4paper,10pt]{article}
\usepackage[utf8]{inputenc}
\usepackage[T1]{fontenc}
\usepackage[utf8]{inputenc}
\usepackage{amsthm}
\usepackage{amsmath}
\usepackage{hyperref}
\usepackage{xspace}
\usepackage{graphicx}
\usepackage[rmargin=3cm,lmargin=3cm,tmargin=2cm,bmargin=3cm]{geometry}
\usepackage{amssymb}
\usepackage{textcomp}
\usepackage{wasysym}
\usepackage{xcolor}

\newtheorem{definition}{Definition}[section]
\newtheorem{proposition}[definition]{Proposition}
\newtheorem{remark}[definition]{Remark}
\newtheorem{theorem}[definition]{Theorem}
\newtheorem{corollary}[definition]{Corollary}

\newtheorem{lemma}[definition]{Lemma}

\newtheorem*{theorem A}{Theorem A}
\newtheorem*{theorem B}{Theorem B}

\numberwithin{equation}{section}

\newcommand{\al}{\ensuremath{\alpha}}
\newcommand{\be}{\ensuremath{\beta}}
\newcommand{\D}{\ensuremath{\Delta}}

\newcommand{\fl}{\ensuremath{\forall}}
\newcommand{\Ga}{\ensuremath{\Gamma}}
\newcommand{\ga}{\ensuremath{\gamma}}
\newcommand{\la}{\ensuremath{\lambda}}
\newcommand{\Le}{\ensuremath{\mathcal{L}^*}}
\newcommand{\na}{\ensuremath{\nabla}}
\newcommand{\pa}{\ensuremath{\partial}}
\newcommand{\pt}{\ensuremath{\partial_t}}
\newcommand{\R}{\ensuremath{\mathbb{R}}}
\newcommand{\ra}{\ensuremath{\rightarrow}}
\newcommand{\si}{\ensuremath{\sigma}}
\newcommand{\bs}{\ensuremath{\backslash}}
\newcommand{\te}{\ensuremath{\theta}}
\title{Harnack inequalities for positive solutions of the heat equation on closed Finsler manifolds}
\author{C. Combete\footnote{Institut de Mathematiques et de Sciences Physiques, Porto-Novo, Benin,\ cyrille.combete@imsp-uac.org}, 
S. Degla\footnote{Ecole normale Superieure de Natitingou, Benin,\ sdegla@imsp-uac.org} 
and  L. Todjihounde\footnote{Institut de Mathematiques et de Sciences Physiques, Porto-Novo, Benin,\ leonardt@imsp-uac.org}}
\date{}

\begin{document}

\maketitle

\begin{abstract}
% The goal of this paper is twofold. Firstly, we generalize some Li-Yau type gradient  estimates to Finsler geometry. 
% Harnack type inequalities are derived. Thereafter, we study 
The main goal of this paper is to generalize some Li-Yau type gradient estimates to Finsler geometry in order to derive Harnack type 
inequalities. Moreover, we obtain, under some curvature assumption, a general gradient estimate for positive solutions of the heat equation 
when the manifold evolving along the Finsler Ricci flow. 
\end{abstract}


%\keywords{Gradient estimates  \and Harnack inequalities \and Finsler Ricci flow \and Heat flow}

\section{Introduction}
A very active research topic in geometric analysis is the study of the heat equation on manifolds because of its several 
applications in physics and natural sciences.
In their famous paper \cite{LY-86}, P. Li and S-T. Yau have proved the following 

\begin{theorem A}\label{theorem A}
 Let $M$ be an $n$-dimensional complete Riemannian manifold with Ricci curvature bounded from below by $-k$ for some nonnegative constant $k$ 
 and let $u\in C^\infty([0,T]\times M)$ be a positive solution of the heat equation
 \begin{equation}
  \D u-\pt u=0.
 \end{equation}
Then, for any $\al>1$, it holds
\begin{equation}\label{eq:li-yau}
 |\na f|^2-\al\pt f \leqslant \frac{n \al}{t}+\frac{n\al^2 k}{2(\al-1)},
\end{equation}
where $f=\log u$. Particularly, when $k=0$, letting $\al\rightarrow 1$, one obtains 
\begin{equation}
 |\na f|^2-\pt f \leqslant \frac{n \al}{t}.
\end{equation}

\end{theorem A}
It is worth to mention that the equation  (\ref{eq:li-yau}) is not sharp unless $k=0$. 
An important application of this estimate is that it provides  Harnack type inequalities.  
There is a rich litterature about improvement and generalization of the Li-Yau gradient estimate (\ref{eq:li-yau}), 
see for example \cite{CN-05,QZZ-13,Quian-14,BBG-17} and the references therein.  These results do not only concern 
the heat equation on Riemannian manifolds but also general linear and semi-linear parabolic equations on Kaeler manifolds 
or Alexandrov spaces.
Despite these important efforts done, the question of sharpness of the Li-Yau gradient estimate remain open.

% \begin{theorem B}
%  Using the same notation as in Theorem A with $M$ closed, let consider the functions $\la, \be, \psi \in C^1((0,T])$ such that 
%  \begin{itemize}
%   \item[$(A_1)$] $\be (t)\in (0,1)$, $\fl\ t \in (0,T]$;
%   \item[$(A_2)$] $\lim_{t\rightarrow 0^+} \la(t)=0$ and $\la(t)>0$, $\fl\ t \in (0,T]$;
%   \item[$(A_3)$] $\frac{2k\be+\be'}{1-\be}-(\ln \la)'>0$ on $(0,T]$; 
%   \item[$(A_4)$] $\limsup_{t\rightarrow 0^+}\psi \geq 0$;
%   \item[$(A_5)$] $\psi'+\frac{2k\be+\be'}{1-\be}\psi-\frac{n(2k\be+\be')^2}{8\be(1-\be)^2}=0$ for any $t \in (0,T]$. 
%  \end{itemize}
% Then, we have 
% \begin{equation}
%  \be |\na f|^2-\pt f \leq \psi \hspace{0.5cm} on\ (0,T].
% \end{equation}
% 
% \end{theorem B}

%Note that this result improve many existing results
Recently, a general gradient estimate for positive solutions of the heat
equation on closed Riemannian manifold which extends many existing estimates was proved   \cite{CYZ-17}.
The first purpose of this paper is to generalize this result to Finsler geometry in order to derive Harnack type inequalities.
Finsler manifolds are natural generalization of Riemannian manifolds in 
the sense that each tangent space is endowed with a Minkowskian norm  instead of the Euclidean one.
Observe that, in this paper, we use the nonlinear Shen's Finsler Laplacian which is an extension among others of the Laplace-Beltrami 
operator to Finsler geometry. Because of the lack of linearity, method used in Riemannian setting does not work. To overcome this difficulty, 
we use the weighted gradient and Laplacian which are linear operators (see Section \ref{sec:02} for the definitions).
Our first result is stated as follows.

\begin{theorem}\label{thm:harnack}
  Let $(M,F)$ be a closed Finsler manifold of dimension $n$, equiped with a smooth measure $d\mu$, such that the weighted 
 Ricci curvature satisfies $Ric_N\geqslant K$ for some $N\in [n,\infty)$ 
 and $K\in \R$ and take a positive global solution $u:[0,T]\times M \rightarrow \R$ to the heat equation. 
 Let us consider the functions $\la, \be, \Psi \in C^1((0,T])$ satisfying
 
  \begin{itemize}
  \item[$(B_1)$] $\be (t)\in (0,1)$, $\fl\ t \in (0,T]$;
  \item[$(B_2)$] $\displaystyle \lim_{t\rightarrow 0^+} \la(t)=0$ and $\la(t)>0$, $\fl\ t \in (0,T]$;
  \item[$(B_3)$] $\displaystyle \frac{\be'-2K^-\be}{1-\be}-(\ln \la)'>0$ on $(0,T]$; 
  \item[$(B_4)$] $\displaystyle\limsup_{t\rightarrow 0^+}\psi \geqslant 0$;
  \item[$(B_5)$] $\displaystyle\Psi'+\frac{\be'-2K^-\be}{1-\be}\Psi-\frac{N(\be'-2K^-\be)^2}{8\be(1-\be)^2}\geqslant 0$ for any $t \in (0,T]$, 
 \end{itemize}
 where $K^-:=\min\{K,0\}$. Then, for any $x,y\in M$ and $0<s<t\leqslant T$, we have
 \begin{equation}
  u(s,x)\leqslant u(t,y) \exp \left\{ \frac{d_F(y,x)^2}{4(t-s)^2}\int_s^t\frac{d\tau}{\be} + \int_s^t \Psi\ d\tau\right\},
 \end{equation}
where $d_F$ is the distance function induced by the Finsler metric $F$.
\end{theorem}
The related  definitions such as weighted Ricci curvature and the Finsler distance function are given in Section \ref{sec:02} below.
Recall that Ohta and Sturm \cite{OS-14} have already proved some Harnack type inequalities in Finsler setting. Theorem \ref{thm:harnack}
can be view as a generalisation of their results.

The Li-Yau type gradient  estimates are also investigated on manifolds evolving along geometric flows. 
The case of the Ricci flow is well known and was iniated by Hamilton in \cite{Ham-82} on 
Riemannian manifolds ; see  \cite{Liu-09,Sun-11} and references therein.
 In \cite{Bao-07}, Bao has introduced the notion of Finsler Ricci flow and recently, Lakzian \cite{Lak-15} 
has derived Harnack estimates for positive solutions to the heat 
equation under this flow. This paper  gives also a generalization of Lakzian result.

\begin{theorem}\label{thm:gradient-flow}
  Let $(M,F)$ be a closed Finsler manifold of dimension $n$ endowed with a smooth measure $d\mu$. 
  Let $(F_t)_{t\in [0,T]}$ be a solution of the Ricci flow on $M$ with $F_0=F$. 
 Assume that there exists some positive constants $K_i$, $i=1,2,3,4$, such that for all $t\in [0,T]$, $F_t$ has isotropic $S$-curvature, 
 $S=\si F_t + d\varphi$ for some time dependent functions $\si$ and $\varphi$ with $-K_3 F_t\leq d\si$ and $-K_4 F_t^2\leq Hess\ \varphi$, 
 and its Ricci curvature  satisfies
 $-K_1\leq Ric_{ij}\leq K_2$. 
 
 Let $u:[0,T]\times M\rightarrow \R$ be a positive solution of the heat equation under the Finsler Ricci flow $(F_t)$. Let 
 $\be, \la \in C^1((0,T])$ satisfying 
 \begin{itemize}
  \item[$(C_1)$] $\be (t)\in (0,1)$, $\fl\ t \in (0,T]$;
  \item[$(C_2)$] $\displaystyle \lim_{t\rightarrow 0^+} \la(t)=0$ and $\la(t)>0$, $\fl\ t \in (0,T]$;
  \item[$(C_3)$] $\frac{2\be'}{1-\be}-(\ln \la)'<0$ on $(0,T]$; 
 \end{itemize}

 Then we have, for any $t\in (0,T]$,
 \begin{equation}
\be F(\na f)^2-\pt f \leq  \frac{n}{2\be}\left((\ln \la)'-\frac{2\be'}{1-\be}\right)+\frac{n(C_1+\be')}{2\be(1-\be)}+
  \frac{n^{3/2}\sqrt{C_2}}{\be}+\sqrt{2nC_3},
 \end{equation}
where $C_1:=K_1$, $C_2:=\max\{K_1^2,K_2^2\}$ and $C_3=K_3+K_4$.
\end{theorem}
%To the authors' knowledge, this result is new even in Riemannian setting.
 We obtain an improvement of \cite[Theorem 1.1]{Lak-15} 
with a suitable choice of functions $\be$ and $\la$.

The paper is organised as follows. In section \ref{sec:02}, we review some facts about Finsler geometry which we need for the sequel.
Thereafter,  we prove  Li-Yau gradient estimate in  section \ref{sect:3}. That allows us to show the Theorem \ref{thm:harnack}.
In the last section, we deal with gradient estimate in time dependent Finsler manifolds where we prove Theorem \ref{thm:gradient-flow}.

\section{Preliminaries}\label{sec:02}
In this section, we briefly recall some  basic concepts of Finsler geometry necessary for further discussions.
We refer to \cite{Ohta-17,Bao-07}  and references therein for more details.

Let $M$ be an $n$-dimensional smooth manifold and $\pi:TM\rightarrow M$ be the natural projection from the tangent bundle $TM$. 
A point $(x,y)\in TM$ is such that $y=y^i\frac{\partial}{\partial x^i}$ in the local coordinates $(x^i,y^i)$ on $TM$. 
Let $F:TM\rightarrow [0,\infty)$ be a Finsler metric on $M$, that is $F$ satisfies 
\begin{itemize}
 \item[(i)] Regularity: $F$ is smooth on $TM\smallsetminus \{0\}$;
 \item[(ii)] Positive homogeneity: $F(x,\la y)=\la F(x,y)$, $\forall\ \la >0$;
 \item[(iii)] Strong convexity: The fundamental quadratic form 
 \begin{equation}\label{eq:fund}
  g=g_{ij}(x,y)\ dx^i\otimes dx^j;\hspace{0.5cm} g_{ij}(x,y):=\frac{1}{2}\frac{\partial^2 F^2}{\partial y^i\partial y^j}(x,y),
 \end{equation}
is positive definite for any $(x,y)\in TM$ with $y\neq 0$.
\end{itemize}

For $x,y\in M$, we define the distance from $x$ to $y$ by 
\begin{equation}
 d_F(x,y):=\inf_\ga \int_0^1 F(\ga(t),\dot\ga(t)\ dt,
\end{equation}
where the infimum is taken over all $C^1$-curves $\ga:[0,1]\rightarrow M$ such that $\ga(0)=x$ and $\ga(1)=y$. 
Since the Finsler metric is only positively homogeneous, the distance function may be asymmetric. 
We call geodesic any $C^\infty$-curve $\ga$ on $M$ which is locally minimising and has constant speed. We also 
define the exponential map by 
\begin{equation*}
 \exp_x(v):= \ga(1),
\end{equation*}
where $\ga:[0,1]\rightarrow M$ is a geodesic with $\dot\ga(0)=v\in T_xM$.

Let $V=v^i\frac{\partial}{\partial x^i}$ be a nonzero vector on an open subset $\mathcal{U}\subset M$.
Through equation (\ref{eq:fund}), one defines a Riemannian metric 
\begin{equation*}
 g_V\left(X^i\frac{\pa}{\pa x^i},Y^j\frac{\pa}{\pa x^j}\right):= g_{ij}(V)X^iY^j,
\end{equation*}
and a covariant derivative by
\begin{equation*}
 D^V_{\frac{\pa}{\pa x^i}}\frac{\pa}{\pa x^j}:=\Ga_{ij}^k(x,V)\frac{\pa}{\pa x^k},
\end{equation*}
where $\Ga_{ij}^k(x,V)$ are coefficients of the Chern connection.

The flag curvature of the plane spanned by two linearly independent vectors $V,W\in T_xM\smallsetminus\{0\}$ is defined by 
\begin{equation*}
 K(V,W):=\frac{g_V(R^V(V,W)W,V)}{g_V(V,V)g_V(W,W)-g_V(V,W)^2},
\end{equation*}
where $R^V$ is the Chern curvature as follows
\begin{equation*}
 R^V(X,Y)Z:=D^V_XD_Y^VZ-D^V_YD^V_XZ-D^V_{[X,Y]}Z.
\end{equation*}

Then, the Ricci curvature of $(M,F)$ is given by
\begin{equation*}
 Ric(V):=\sum_{i=1}^{n-1}K(V,e_i),
\end{equation*}
where $\{e_1,\dots,e_{n-1},e_n:=\frac{V}{F(V)}\}$ is an orthonormal basis of $T_xM$ 
with respect to $g_V$ and the Ricci tensor is defined as follows :
\begin{equation}
 Ric_{ij}:=\frac{1}{2}\frac{\pa^2 (F^2Ric)}{\pa y^i\pa y^j}.
\end{equation}

Now, let us consider an arbitrary volume form $d\mu=\sigma(x)dx$ on $M$. 
The distortion of $(M,F,d\mu)$ is defined, for any $y\in T_xM\smallsetminus \{0\}$, by 
\begin{equation*}
 \tau(x,y):=\ln \left( \frac{\sqrt{det(g_{ij}(x,y))}}{\sigma(x)}\right).
\end{equation*}

Let $y\in T_xM\smallsetminus \{0\}$ and $\ga:(-\epsilon,\epsilon)\longrightarrow \R$ be the geodesic with $\ga(0)=x$ and $\dot\ga(0)=y$.
The $S$-curvature  defined as 
\begin{equation*}
 S(x,y):=\frac{d}{dt}\left[\tau\big(\dot \ga(t)\big)\right]_{\big|t=0},
\end{equation*}
 measures the rate of changes of the distortion along geodesics. One also defines
 \begin{equation*}
  \dot S(y):=\frac{1}{F(x,y)^2}\frac{d}{dt}\left[S\big(\ga(t),\dot\ga(t)\big)\right]_{\big|t=0}.
 \end{equation*}

The weighted Ricci curvature of the Finsler manifold $(M,F,d\mu)$ is defined as follows %(\cite{HY-15})
\begin{eqnarray*}
 Ric_n(y)&:=& \left\{ \begin{array}{l}
               Ric(y)+\dot S(y),\ \text{ for }\ S(y)=0,\\[.2cm]
               -\infty,\hspace{0.5cm} \text{ otherwise},
              \end{array}
               \right.\\
Ric_N(y)&:=& Ric(y)+\dot S(y)-\frac{S(y)^2}{(N-n)F(y)^2},\ \forall\ N\in(n,\infty)\\
Ric_\infty(y)&:=& Ric(y)+\dot S(y).
\end{eqnarray*}

\vspace{0.5cm}

In  analogy with the Ricci flow in Riemannian setting, Bao \cite{Bao-07} proposed the following 
definition of the Finsler Ricci flow:
\begin{equation}
 \frac{\pa}{\partial t}g_{ij}=-2Ric_{ij},\quad g_{|t=0}=g_0,
\end{equation}
which contrating to $y^iy^j$, via Euler theorem, provides 
\begin{equation}
 \frac{\pa}{\partial t}\log F=-Ric,\quad  F_{|t=0}=F_0, 
\end{equation}
where $F_0$ is the initial Finsler structure. Short time existence and uniqueness of the Finsler Ricci flow 
has been studied in some particular cases, see e.g \cite{AR-13,BS-15}.
\vspace{0.5cm}

For their convenient use, let us recall the definitions of the gradient, the Hessian and the Laplacian operators on %the Finsler manifold 
$(M,F,d\mu)$. Let $u$ be a differentiable function  and  $V$ a 
smooth vector field  on $M$. We set $M_u:=\{x\in M ; \ du(x)\neq 0\}$ and $M_V:=\{x\in M\ ; \ V(x)\neq 0\}$. 
The divergence of $V:=V^i\frac{\pa}{\pa x^i}$ with respect to the volume form $d\mu:=\sigma(x)dx$ is defined by
\begin{equation*}
 div(V):=\sum_{i=1}^n\left(\frac{\pa V^i}{\pa x^i}+\frac{V^i}{\sigma}\frac{\pa \sigma}{\pa x^i}\right).
\end{equation*}

Let $\Le:T^*M\longrightarrow TM$ be the Legendre transform which assigns to each $\al\in T_x^*M$ the unique element $v\in T_xM$ such that 
$\al(v)=F^*(x)^2$ and $F(v)=F^*(\al)$, where $F^*$ stands for the dual norm of $F$. Then the gradient vector and the Laplacian
of $u$ are given by 
\begin{equation*}
 \na u(x)=\Le\big(du(x)\big),\hspace{0.5cm} \D u(x):=div(\na u).
\end{equation*}
If $V$ does not vanish on $M_u$, one can also define  the weighted gradient vector and the weighted Laplacian of $u$ 
on the Riemannian manifolds $(M,g_V)$ as
\begin{equation*}
 \na^V u:=\left\{ \begin{array}{l}
                   g_{ij}(V)\frac{\pa u}{\pa x^j}\frac{\pa }{\pa x^i},\ \text{ on }\ M_V,\\[.2cm]
                   0\hspace{1cm} \ \text{ on } M\smallsetminus  M_V,
                  \end{array}
           \right.,\hspace{0.5cm}
           \D^V u:=div(\na^V u).
\end{equation*}
In particular, on $M_u$, we have
\begin{equation*}
 \na u=\na^{\na u}u,\hspace{0.5cm} \D u=\D^{\na u}u.
\end{equation*}
In other words, the Laplacian of $u$ can be expressed using the $S$-curvature as 
\begin{equation}
 \D u=tr_{g_{\na u}}(\na^2u)-S(\na u),
\end{equation}
where $\na^2u:=D^{\na u}_.\na u$ is the Hessian of $u$. Here the trace is taken with respect to an $g_{\na u}$-orthonormal basis. 
The point-wise Finslerian version of the Bochner-Weitzenbock formulas are given by 
\begin{equation}
 \D^{\na u}(\frac{F^2(\na u)}{2})-d(\D u)(\na u)\geq Ric_N(\na u)+\frac{(\D u)^2}{N},
\end{equation}
\begin{equation}
  \D^{\na u}\left(\frac{F^2(\na u)}{2}\right)-d(\D u)(\na u)=Ric_\infty(\na u)+\|\na^2u\|^2_{HS(\na u)}.
\end{equation}

We conclude this section by recalling some notions about the heat equation $\pt u=\D u$ associated with 
the Finsler Laplacian. We consider a Finsler space $(M,F,d\mu)$ where $(M,F)$ is a Finsler space 
and $d\mu$ a smooth measure on $M$.

A function $u:[0,T]\times M\rightarrow \R$ is said to be a global solution of the heat equation $\pt u=\D u$ if it satifies the 
following:
\begin{itemize}
 \item[(i)] $u\in L^2\big([0,T],H^1_0(M)\big)\ \cap\ H^1([0,T],H^-1(M)) $,
 \item[(ii)] for all $t\in [0,T]$ and $\phi\in C_0^\infty(M)$,
 \begin{equation*}
  \int_M \phi.\pt u_t\ d\mu=-\int_M d\phi(\na u_t)\ d\mu,
 \end{equation*}
where  $u_t:=u(t,.)$.
\end{itemize}

\section{Gradients estimates on static Finsler manifolds}\label{sect:3}

This section deals with the proof of Theorem \ref{thm:harnack}.
Let $(M,F,d\mu)$ be a closed $n$-dimensional Finsler  manifold equiped with a smooth measure $d\mu$. 
Assume that its weighted Ricci curvature $Ric_N\geq K$ for some $N\in [n,\infty)$ and $K\in \R$.

Let us consider  a global positive solution 
$u:[0,T]\times M\rightarrow \R$ to the Finsler heat equation on $(M,F,d\mu)$. 
We will fix a measurable one-parameter family of non-vanishing vector fields 
$\{V_t\}_{t\in [0,T]}$ on $M$ such that $V_t=\na u_t$. %on $M_{u_t}:=\{x\in M|\ u_t(x)\neq 0\}$. 
Here, $u_t(x):=u(t,x)$, $(t,x)\in [0,T]\times M$. 

For any  $(t,x)\in [0,T]\times M$, we set  $f(t,x):=f_t(x)=\ln u(t,x)$. 
One can easily check that   (see \cite[Equation 4.2]{OS-14}),
\begin{equation}\label{eq:p-log}
 \pt\big(F(\na f)^2\big)=2d(\pt f)(\na f),
\end{equation}
and for every $t\in [0,T]$, we have in distributional sense,
\begin{equation}\label{eq:l-log}
 \D f+ F(\na f)^2=\pt f.
\end{equation}

Suppose further that there exist some functions $\la, \be, \Psi \in C^1((0,T])$ such that
  \begin{itemize}
  \item[$(B_1)$] $\be (t)\in (0,1)$, $\fl\ t \in (0,T]$;
  \item[$(B_2)$] $\displaystyle \lim_{t\rightarrow 0^+} \la(t)=0$ and $\la(t)>0$, $\fl\ t \in (0,T]$;
  \item[$(B_3)$] $\displaystyle \frac{\be'-2K^-\be}{1-\be}-(\ln \la)'>0$ on $(0,T]$; 
  \item[$(B_4)$] $\displaystyle\limsup_{t\rightarrow 0^+}\psi \geqslant 0$;
  \item[$(B_5)$] $\displaystyle\Psi'+\frac{\be'-2K^-\be}{1-\be}\Psi-\frac{N(\be'-2K^-\be)^2}{8\be(1-\be)^2}\geqslant 0$ for any $t \in (0,T]$, 
 \end{itemize}
 where $K^-:=\min\{K,0\}$.

Let us consider the function $G$ defined on $[0,T]\times M$ by $G=\be F(\na f)^2-\pt f-\Psi$.
We claim the following

\begin{lemma}
 $G$ satisfies
\begin{equation}\label{eq:lG}
 \D^V G + 2 dG(\na f) -\pt G \geqslant \frac{\be'-2K^-\be}{1-\be} G,
\end{equation}
in the sense of distribution on $(0,T)$.
\end{lemma}

\begin{proof}
 For all $\phi \in H^1_0\big((0,T)\times M\big)$, we have 
\begin{eqnarray*}
  & & \int_0^T \!\!\!\int_M \{-d\phi(\na^V (\pt f)) + 2 \phi d(\pt f))(\na f) -\phi\pt(\pt f)\}\ d\mu dt\\
     &=& \int_0^T\!\!\! \int_M \{-d\phi(\na^V (\pt f)) + \phi [ 2  d(\pt f))(\na f)-\pt(F(\na f)^2)-\pt(\D f)]\}\ d\mu dt\\
     &=& \int_0^T \!\!\!\int_M \{-d\phi(\na^V (\pt f)) -\phi \D(\pt f)\}\ d\mu dt\\
     &=& 0,
\end{eqnarray*}
where we used equations (\ref{eq:l-log}) and (\ref{eq:p-log}).

Furthermore, using again (\ref{eq:l-log}), (\ref{eq:p-log}) and the Bochner-Weitzenbock inequality, 
we have for any nonnegative  test function $\phi \in H^1_0\big((0,T)\times M\big)$,

\begin{eqnarray*}
  & & \int_0^T \!\!\!\int_M \{-d\phi(\na^V (F(\na f)^2)) + 2 \phi d(F(\na f)^2))(\na f) -\phi\pt(F(\na f)^2)\}\ d\mu dt\\
    &=& \int_0^T\!\!\! \int_M \{-d\phi(\na^V (F(\na f)^2)) -2\phi d(\D f)(\na f) +\phi(2d(\pt f)(\na f)-\pt(F(\na f)^2))\}\ d\mu dt\\
    &=& \int_0^T\!\!\! \int_M \{-d\phi(\na^V (F(\na f)^2)) -2\phi d(\D f)(\na f)\}\ d\mu dt\\
    &\geqslant& \int_0^T\!\!\! \int_M 2 \phi \{Ric_N(\na f)+\frac{(\D f)^2}{N}\}\ d\mu dt 
    \geqslant \int_0^T\!\!\! \int_M 2 \phi \{K F(\na f)^2 +\frac{(\D f)^2}{N}\}\ d\mu dt\\
    &\geqslant& \int_0^T\!\!\! \int_M 2 \phi \{K^- F(\na f)^2 +\frac{(\D f)^2}{N}\}\ d\mu dt\\
\end{eqnarray*}

Therefore, for every nonnegative $\phi \in H^1_0\big((0,T)\times M\big)$,
\begin{eqnarray*}
  & & \int_0^T\int_M \left\{ -d\phi (\na^V G)+2\phi dG(\na f)-\phi \pt G \right\}\ d\mu dt \\
   &\geqslant& \int_0^T\int_M \phi\left\{2 \be\left(K^- F(\na f)^2 +\frac{(\D f)^2}{N} \right) -\be'F(\na f)^2+\Psi'\right\}\ d\mu dt\\
   &=& \int_0^T\int_M \phi\left\{ \frac{2\be}{N}(\D f)^2+\frac{\be'-2K^-\be}{1-\be}(\D f+G+\Psi) + \Psi'\right\}\ d\mu dt\\
   &\geqslant& \int_0^T\int_M \phi\left\{ \frac{\be'-2K^-\be}{1-\be}G+\Psi'+\frac{\be'-2K^-\be}{1-\be}\Psi
   -\frac{N(\be'-2K^-\be)^2}{8\be(1-\be)^2}\right\}\ d\mu dt\\
   &\geqslant& \int_0^T\int_M \frac{\be'-2K^-\be}{1-\be}\phi G \ d\mu dt
\end{eqnarray*}

which complete the proof of (\ref{eq:lG}). 
\end{proof}

Now, let us define the function $H(t,x)=\la(t)G(t,x)$. From (\ref{eq:lG}), we can easily deduce that the following equation 
\begin{equation}\label{eq:lH}
 \D^VH+2dH(\na f)-\pt H\geqslant \left(\frac{\be'-2K^-\be}{1-\be}-(\ln \la)' \right) H,
\end{equation}
holds in the distributional sense on $(0,T)$.

Consider a point $(t_0,x_0)$ at which $H$ attains its maximum on $[0,T]\times M$.

\begin{lemma}
 We have $H(t_0,x_0)\leqslant 0$.
\end{lemma}

\begin{proof}
 Assume by contradiction that $H(t_0,x_0)>0$. From conditions $(B_1)$, $(B_2)$ and $(B_4)$, we can deduce that 
$\displaystyle\liminf_{t\rightarrow0^+}H\leqslant 0$ and thus $t_0>0$. On other hand, by $(B_3)$ ,
\begin{equation}
 \left(\frac{\be'(t_0)-2K^-\be(t_0)}{1-\be(t_0)}-(\ln \la)'(t_0) \right) H(t_0,x_0)>0
\end{equation}
on some neighborhood of $(t_0,x_0)$. Therefore, $H$ is a strict subsolution of the linear parabolic operator 
\begin{equation}
 div_\mu(\na^V H)+2dH(\na f)-\pt H, 
\end{equation}
on such a neighborhood. 

This implies that $H(t_0,x_0)$ is trictly less than the supremum of $H$
on the boundary of any small parabolic cylinder $[t_0-\delta, t_0]\times B_\delta(x_0)$, where 
$B_\delta(x_0)=\{x\in M\ ; \ d_F(x_0,x)<\delta\}$. Hence $(t_0,x_0)$ cannot be the maximum point of $H$.
\end{proof}

Hence, we have proved the following gradient estimate:
\begin{theorem}\label{thm:grad-esti}
 Let $(M,F,d\mu)$ be a closed Finsler space of dimension $n$ with weighted 
 Ricci curvature satisfying $Ric_N\geq K$ for some $N\in [n,\infty)$ and 
 $K\in \R$. Let $u:[0,T]\times M\ra \R$ be a positive solution of the associated heat equation. If there are function 
 $\la, \be, \Psi \in C^1((0,T])$ satisfying assumptions $(B_1)-(B_5)$, then for any $t\in (0,T)$, we have 
 \begin{equation}
  \be F(\na f)^2-\pt f \leq \psi,
 \end{equation}
 where $f:=\ln u$.
\end{theorem}

Following \cite{CYZ-17}, we obtain this immediate consequence of Theorem \ref{thm:grad-esti}:

\begin{corollary}
 Let $(M,F,d\mu)$ and $u$ be as in Theorem \ref{thm:grad-esti} with $K<0$. 
 Let $b\in C^1\big((0,T]\big)$ be a positive increasing function on $(0,T]$ such that 
 \begin{equation*}
  \lim_{t\rightarrow 0^+}b(t)=0\ and\ \frac{b'^2}{b}\in L^1\big([0,T]\big).
 \end{equation*}
Then,
\begin{equation*}
 \be F(\na f)-\pt f\leqslant \Psi,
\end{equation*}
where 
\begin{equation}
 \be = 1+\frac{2K}{b(t) e^{-2Kt}}\int_0^tb(s)e^{-2Ks}\ ds,
\end{equation}
and 
\begin{equation}
 \Psi=\frac{n}{8b}\int_0^t\frac{b'^2}{b\be}(s)\ ds.
\end{equation}

\end{corollary}

\begin{remark}
 Taking $b(t)=(1-\theta K t)t^{\frac{2}{\theta}-1}$ and $b(t)=\sinh^2(-Kt)+\cosh(-Kt)\sinh(-Kt)+Kt$, one obtains respectively 
 \begin{equation*}
  F(\na f)^2-(1-\theta K t)\pt f \leqslant \frac{n(2-\theta)^2}{16\theta(1-\theta)t}+\frac{nK^2\theta t}{4}-\frac{nK}{2},
 \end{equation*}

 and
 \begin{equation*}
  F(\na f)^2-\left(1+\frac{\sinh(-Kt)\cosh(-Kt)-Kt}{\sinh^2(-Kt)}\right)\pt f \leqslant \frac{-nK}{2}(\coth(-Kt)+1).
 \end{equation*}

\end{remark}

We can now derive our Harnack type inequality.

\begin{proposition}
 Under assumptions of Theorem \ref{thm:grad-esti}, we have, for any $x,y\in M$ and $0<s<t\leqslant T$, 
 \begin{equation}
  u(s,x)\leqslant u(t,y) \exp \left\{ \frac{d_F(y,x)^2}{4(t-s)^2}\int_s^t\frac{d\tau}{\be} + \int_s^t \Psi\ d\tau\right\},
 \end{equation}
where $d_F$ is the distance function induced by the Finsler metric $F$.
\end{proposition}

\begin{proof}
  Consider the reverse curve $\eta:[s,t]\ni \tau \mapsto \eta(\tau):=\exp_y((t-\tau)v)\in M$ of the minimal geodesic joining 
 $y=\eta(t)$ to $x=\eta(s)$ where $v\in T_yM$ is a suitable vector. For any $\tau \in [s,t]$, we have $F(-\dot\eta(\tau))=d_F(y,x)/(t-s)$. 
 Let $f=\ln u$ and $\sigma(\tau)=-f(\tau,\eta(\tau))$. Then, according to Theorem \ref{thm:grad-esti}, we have
 \begin{eqnarray*}
  \sigma(t)-\sigma(s) &=& \int_s^t \dot\sigma(\tau)\ d\tau\\
                         &=& \int_s^t \{df(-\dot\eta)-\pt f\}\ d\tau\\
                         &\leqslant& \int_s^t \{ F(\na f)F(-\dot \eta)+\Psi-\be F(\na f)^2\}\ d\tau\\
                         &=& \int_s^t \{ F(\na f)\frac{d_F(y,x)}{t-s} + \Psi - \be F(\na f)^2\}\ d\tau\\
                         &\leqslant& \int_s^t \left\{\frac{1}{4\be}\frac{d_F(y,x)^2}{(t-s)^2}+\Psi \right\}\ d\tau 
 \end{eqnarray*}
where we used in the last line the fact that the polynome $P(x)=-\be x^2+\frac{d_F(y,x)}{t-s} x+ \Psi$ attains its maximum at 
$x_0=\frac{d_F(y,x)}{2\be(t-s)}$.
\end{proof}

\section{Gradient estimates with time dependant Finsler metrics}

 In this section, we consider a smooth manifold evolving along the Finsler Ricci flow. Our aim is 
 to prove a general gradient estimate for  positive solutions of the heat equation  under the Ricci flow, that is 
 functions $u\in L^2([0,T],H^1(M))\cap H^1([0,T],H^{-1}(M))$ satisfying 
  \begin{equation*}
  \int_M \phi.\pt u_t\ d\mu=-\int_M d\phi(\na u_t)\ d\mu,
  \end{equation*}
  where the gradient is with respect to the Finsler metric $F_t$. Let us observe that, here we didn't investigated about existence and 
  regularity of such solutions.

%we will prove Theorem \ref{thm:gradient-flow} in oder to derive Harnack type estimates.

Let $(M,F)$ be a closed $n$-dimensional Finsler manifold endowed with a smooth mesure $d\mu$.
Let $(F_t)_{t\in[0,T]}$ be a solution of the Ricci flow on $M$ with $F_0=F$. Let us consider a positive 
solution $u:[0,T]\times M\rightarrow \R$ of the heat equation under the Ricci flow, that is 
$u\in L^2([0,T],H^1(M))\cap H^1([0,T],H^{-1}(M))$ and satisfies 
 \begin{equation*}
  \int_M \phi\cdot\pt u_t\ d\mu=-\int_M d\phi(\na u_t)\ d\mu,
 \end{equation*}
 where the gradient is taken with respect to the Finsler metric $F_t$ and let $f=\ln u$.

We have the following evolution equation
\begin{lemma}\cite{Lak-15}
Along the Finsler Ricci flow $(F_t)_{t\in[0,T]}$, it holds
 \begin{equation}
 \pt(F(\na f)^2)= 2d(\pt f)(\na f)+ 2Ric^{ij}(\na f)f_if_j.
\end{equation}
\end{lemma}

Let $\be, \la \in C^1((0,T])$ be three functions satisfying the following assumptions:
 \begin{itemize}
  \item[$(C_1)$] $\be (t)\in (0,1)$, $\fl\ t \in (0,T]$;
  \item[$(C_2)$] $\displaystyle \lim_{t\rightarrow 0^+} \la(t)=0$ and $\la(t)>0$, $\fl\ t \in (0,T]$;
  \item[$(C_3)$] $\displaystyle\frac{2\be'}{1-\be}-(\ln \la)'<0$ on $(0,T]$. 
 \end{itemize}

On $[0,T]\times M$, define $G:=\be F(\na f)^2-\pt f$ and $H:=\la G$. 

\begin{lemma}
 It holds
 \begin{equation}\label{eq:H}
  \D^V H+ 2dH(\na f)-\pt H= \left(\frac{\be'}{1-\be}-(\ln \la)'\right)H + \la I
 \end{equation}
 in the sense of distribution on $(0,T)\times M$, where 
\begin{equation}
 I(t,x):=\frac{\be'}{1-\be}\D f+ 2(1-\be)Ric^{ij}(\na f)f_if_j +2\be(Ric(\na f)+\dot S(\na f)+\|\na^2 f\|^2)+2Ric^{ij}(\na f)f_{ij}
\end{equation}
\end{lemma}

\begin{proof}
  
We have
\begin{eqnarray}
\pt(d\phi(\na f))&=& d(\pt \phi)(\na f)+ d\phi(\pt(\Le (df)))\nonumber \\
                 &=& d(\pt \phi)(\na f)+ d\phi(\Le(d(\pt f)))+ d\phi((\pt\Le)(\na f))\nonumber\\
                 &=& d(\pt \phi)(\na f)+ d\phi(\na(\pt f))+ 2 g^{jk} Ric^i_k(\na f) \phi_if_j\nonumber\\
                 &=& d(\pt \phi)(\na f)+ d\phi(\na(\pt f))+ 2 Ric^{ij}(\na f)\phi_if_j
\end{eqnarray}

Using this relation, we compute
 
 \begin{eqnarray}
   & &\int_0^T\!\!\!\int_M\{ -d\phi(\na^V(\pt f))+ 2\phi d(\pt f)(\na f)-\phi \pt(\pt f)\}\ d\mu dt\nonumber \\
          &=& \int_0^T\!\!\!\int_M\{ -d\phi(\na^V(\pt f)) + \phi [2 d(\pt f)(\na f) -\pt(F(\na f)^2)-\pt(\D f)]\}\nonumber d\mu dt\\
          &=& \int_0^T\!\!\!\int_M\{ -d\phi(\na^V(\pt f)) - \phi [2 Ric^{ij}(\na f)f_if_j+\pt(\D f)]\}\ d\mu dt\nonumber \\
          &=& \int_0^T\!\!\!\int_M\{ -\pt(d\phi(\na f))+d(\pt \phi)(\na f)+ 2 Ric^{ij}(\na f)\phi_if_j\nonumber\\
          &-&  \phi [2 Ric^{ij}(\na f)f_if_j+\pt(\D f)]\}\ d\mu dt\nonumber\\
          &=& \int_0^T\!\!\!\int_M\{ -\pt(d\phi(\na f)) + 2 Ric^{ij}(\na f)\phi_if_j- 2\phi Ric^{ij}(\na f)f_if_j\}\ d\mu dt\nonumber\\
          &=& \int_0^T\!\!\!\int_M\{ 2 Ric^{ij}(\na f)\phi_if_j- 2\phi Ric^{ij}(\na f)f_if_j\}\ d\mu dt\nonumber\\
          &=& \int_0^T\!\!\!\int_M -2\phi \{ Ric^{ij}(\na f)f_{ij} + Ric^{ij}(\na f)f_if_j\}\ d\mu dt,
 \end{eqnarray}
 
 and
 
 \begin{eqnarray*}
  & & \int_0^T\!\!\!\int_M\{ -d\phi(\na^V(F(\na f)^2))+ \phi [2d(F(\na f)^2)(\na f)- \pt(F(\na f)^2)]\}\ d\mu dt\nonumber \\
  &=& \int_0^T\!\!\!\int_M\{ -d\phi(\na^V(F(\na f)^2))+ 2\phi [d(F(\na f)^2)(\na f)-d(\pt f)(\na f)-Ric^{ij}f_if_j]\}\nonumber d\mu dt\\
  &=& \int_0^T\!\!\!\int_M\{ -d\phi(\na^V(F(\na f)^2))- 2\phi [d(\D f)(\na f)+Ric^{ij}f_if_j]\}\nonumber d\mu dt.
 \end{eqnarray*}

 %Since we assume that for each $t\in [0,T]$, the $S$-curvature vanishes, 
 From the Bochner-Weitzenbock formula, 
 we have 
 \begin{eqnarray}
  & & \int_0^T\int_M\{ -d\phi(\na^V(F(\na f)^2))+ \phi [2d(F(\na f)^2)(\na f)- \pt(F(\na f)^2)]\}\ d\mu dt\nonumber \\
  &=&  \int_0^T\int_M 2\phi\{Ric(\na f)+\dot S(\na f)+\|\na^2 f\|^2_{HS(\na f)}-Ric^{ij}(\na f)f_if_j\} \ d\mu dt.
 \end{eqnarray}

Hence
 \begin{eqnarray}
  & & \int_0^T\!\!\!\int_M\{ -d\phi(\na^VG)+ 2\phi dG(\na f)-\phi \pt(G)\}\ d\mu dt\nonumber \\
  &=& \int_0^T\!\!\!\int_M \phi \Big\{-\be'F(\na f)^2+ 2\be\left(Ric(\na f)+\dot S(\na f)+\|\na^2 f\|^2_{HS(\na f)}-Ric^{ij}(\na f)f_if_j)\right)\nonumber\\
  & &  +2 \left( Ric^{ij}(\na f)f_{ij} + Ric^{ij}(\na f)f_if_j\right) \Big\}\ d\mu dt\nonumber\\
  &=& \int_0^T\!\!\!\int_M \phi \Big\{ \frac{\be'}{1-\be}(G+\D f)+ 2(1-\be)Ric^{ij}(\na f)f_if_j\nonumber\\ 
  & & +2\be(Ric(\na f)+\dot S(\na f)+\|\na^2 f\|^2_{HS(\na f)}) +2Ric^{ij}(\na f)f_{ij}\Big\}\ d\mu dt.
 \end{eqnarray}

 Therefore,
 \begin{eqnarray}
  & & \int_0^T\!\!\!\int_M\{ -d\phi(\na^V H)+ 2\phi dH(\na f)-\phi \pt(H)\}\ d\mu dt\nonumber \\
  &=&  \int_0^T\!\!\!\int_M \left\{-\phi\la'G +\la\left(-d\phi(\na^VG)+ 2\phi dG(\na f)-2\phi \pt(G)\right)\right\}\ d\mu dt.
 \end{eqnarray}
\end{proof}

Now, we assume that there exists some positive constant $K_1,K_2,K_2,K_4$ such that for all $t\in[0,T]$, 
\begin{itemize}
 \item  the Ricci curvature of $(M,F_t)$ satifies in any $g_v$-orthonormal basis,
 \begin{equation}
  -K_1\leq (Ric_{ij}(v))_{i,j=1}^n\leq K_2,
 \end{equation}
 as a quadratic form on $T_xM$, for all $v\in T_xM\bs\{0\}$, %in any coordinate system $(x^i)$
 \item $(M,F_t,d\mu)$ has isotropic $S$-curvature, $S_t:=\si F_t + d\varphi$, 
where $\si$ and $\varphi$ are time dependent functions which verify 
\begin{equation}\label{eq:is-curv}
 -K_3 F_t\leq d\si_t \hspace{1cm} and \hspace{1cm} -K_4 F^2_t\leq Hess\ \varphi_t.
\end{equation}
%for any $t\in [0,T]$.
\end{itemize}

\begin{lemma}
  The function $I$ satisfies
  \begin{eqnarray}\label{eq:I}
   I(t,x)&\geqslant& \frac{\be}{n}\left(F(\na f)^2-\pt f\right)^2 + \frac{\be'}{1-\be}\left(F(\na f)^2-\pt f\right)\nonumber\\
         &-& C_1 F(\na f)^2-\frac{n^2C_2}{\be}-2\be C_3,
  \end{eqnarray}
  with $C_1=K_1$, $C_2=\max\{K_1^2,K_2^2\}$ and $C_3=K_3+K_4$.
\end{lemma}

\begin{proof}
   By Young's inequality, we have
  \begin{equation}\label{eq:young}
   Ric^{ij}f_{ij}\leqslant \frac{1}{2\be}\sum_{i,j}(Ric^{ij})^2+\frac{\be}{2} \sum_{i,j} f_{ij}^2. 
  \end{equation}
Choose a nornal coordinate system $\{ x^i\}$ with respect to $g_{\na f}$ such that 
$\na f(x)=\frac{\partial}{\partial {x^1}}$ and for all $i,j$, $\Gamma^{1}_{ij}(\na f(x))=0$ . 
Therefore, one can easily check the following
\begin{equation}\label{eq:laplace-local}
 Ric^{ij}(\na f)=Ric_{ij}(\na f),\hspace{0.5cm}\|\na^2f\|_{HS(\na f)}=\sum_{i,j=1}^nf_{ij}^2,\quad \text{and} \hspace{0.5cm} \sum_{i=1}^n f_{ii}=\D f(x). 
\end{equation}
We have also
\begin{equation}\label{eq:laplace-arithm}
 \sum_{i,j=1}^nf_{ij}^2 \geqslant \sum_{i=1}^n f_{ii}^2 \geqslant \frac{1}{n}\left(\sum_{i=1}^n f_{ii}\right)^2= \frac{1}{n}(\D f)^2
\end{equation}

Using (\ref{eq:young})-(\ref{eq:laplace-arithm}) with estimates on hypothesis of the Theorem, we get 
\begin{eqnarray*}
 I(t,x)&=& \frac{\be'}{1-\be}\D f+ 2(1-\be)Ric^{ij}(\na f)f_if_j \\
       &  & +2\be(Ric(\na f)+\dot S(\na f)+\|\na^2 f\|^2)+2Ric^{ij}(\na f)f_{ij}\\
       &\geqslant& \frac{\be'}{1-\be}\D f+ 2(1-\be)(-K_1F(\na f)^2)\\
       &+ &2\be\left(-K_1F(\na f)^2-K_3-K_4+\sum_{i,j}f_{ij}^2\right)-\frac{1}{\be}\sum_{i,j}(Ric^{ij})^2-\be \sum_{i,j} f_{ij}^2\\
       &=& \frac{\be'}{1-\be}\D f -2K_1 F(\na f)^2+\be \sum_{i,j}f_{ij}^2-2\be(K_3+K_4) -\frac{1}{\be}\sum_{i,j}(Ric^{ij})^2\\
       &\geqslant& \frac{\be'}{1-\be}\D f -2C_1 F(\na f)^2 +\frac{\be}{n}(\D f)^2-2\be C_3 -\frac{1}{\be}n^2C_2,
\end{eqnarray*}
where we set $C_1=K_1$, $C_2=\max\{K_1^2,K_2^2\}$ and $C_3=K_3+K_4$.
\end{proof}

 Combining (\ref{eq:H})  and (\ref{eq:I}),  we obtain 
 \begin{equation}
  \D^V H+ 2dH(\na f)-\pt H\geqslant \left(\frac{\be'}{1-\be}-(\ln \la)'\right)H + \la J,
 \end{equation}
in the distributional sense on $(0,T)\times M$, where 
\begin{eqnarray*}
 J(t,x)&:=&\frac{\be}{n}\left(F(\na f)^2-\pt f\right)^2 + \frac{\be'}{1-\be}\left(F(\na f)^2-\pt f\right)
       - 2C_1 F(\na f)^2-\frac{n^2C_2}{\be}-2\be C_3.
\end{eqnarray*}
Let $(t_0,x_0)$ be the maximizer of $H$ on $[0,T]\times M$. Without lost of generality, we can assume $H(t_0,x_0)>0$,
since, otherwise Theorem \ref{thm:gradient-flow} is trivialy satisfed.
 From condition $(C_2)$, we have $t_0>0$. Then, at $(t_0,x_0)$ we have necessary 
 \begin{equation}\label{eq:HJ}
  \left(\frac{\be'}{1-\be}-(\ln \la)'\right)H + \la J\leqslant 0.
 \end{equation}
Indeed, if we assume the contrary, by an argument analogue to the proof of Theorem \ref{thm:grad-esti}, 
one can show that $H(t_0,x_0)$ could not be the supremum of $H$, that is a contradiction. 

 At $(t_0,x_0)$, we have 
 \begin{equation}\label{eq:H-final}
  H\leqslant \la \left[ \frac{n}{2\be}\left((\ln \la)'-\frac{2\be'}{1-\be}\right)+\frac{n(C_1+\be')}{2\be(1-\be)}+
  \frac{n^{3/2}\sqrt{C_2}}{\be} + \sqrt{2nC_3}\right].
 \end{equation}
 Indeed, let $w:=F(\na f)^2$ and $z:=\pt f$. We have 
\begin{eqnarray*}
 (w-z)^2&=&(\be w-z)^2+ 2(1-\be)w(\be w-z) +(1-\be)^2w^2\\
        &=& \frac{H^2}{\la^2}+2(1-\be)w\frac{H}{\la}+(1-\be)^2w^2.
\end{eqnarray*}
Then,
\begin{eqnarray*}
J&=& \frac{\be}{n}(w-z)^2+\frac{\be'}{1-\be}(w-z)-2C_1w-\frac{n^2C_2}{\be}-2\be C_3\\
 &=&\frac{\be}{n}\left[ \frac{H^2}{\la^2}+2(1-\be)w\frac{H}{\la}+(1-\be)^2w^2\right]+\frac{\be'}{1-\be}\frac{H}{\la}
    -\be'w-2C_1w-\frac{n^2C_2}{\be}-2\be C_3\\
 &=&\frac{\be}{n} \frac{H^2}{\la^2}+2\frac{\be}{n}(1-\be)w\frac{H}{\la} +\frac{\be'}{1-\be}\frac{H}{\la} +\frac{\be}{n}(1-\be)^2w^2
    -(2C_1+\be')w-\frac{n^2C_2}{\be}-2\be C_3\\
 &\geqslant& \frac{\be}{n} \frac{H^2}{\la^2}+\frac{\be'}{1-\be}\frac{H}{\la}-\frac{n(2C_1+\be')^2}{4\be(1-\be)^2}
 -\frac{n^2C_2}{\be}-2\be C_3,
\end{eqnarray*}
where used have been made of $ax^2+bx\geqslant -\dfrac{b^2}{4a}$, $\forall\ x\in \R$ for $a>0$, in the last line. 

Replacing this estimate in (\ref{eq:HJ}) yields 
at $(t_0,x_0)$,
\begin{equation}\label{eq:H-ineq}
 \frac{\be}{n}H^2+\la\left(\frac{2\be'}{1-\be}-(\ln \la)'\right)H-\la^2 \left(\frac{n(2C_1+\be')^2}{4\be(1-\be)^2}
 -\frac{n^2C_2}{\be}-2\be C_3\right) \leqslant 0.
\end{equation}
Remark that the left hand side of  (\ref{eq:H-ineq}) is a quadratic polynomial $P(x)=a x^2+b x+c$ in $H$ with $a>0$, $b< 0$ and $c<0$. 
Hence, using $x \leqslant \frac{-b+\sqrt{b^2-4ac}}{2a}\leqslant \frac{-b+\sqrt{-4ac}}{2a}$, one obtains (\ref{eq:H-final}).

\vspace{0.25cm}
This completes the proof of Theorem \ref{thm:gradient-flow} wich we state again  here

\begin{theorem}\label{thm:gradient-flow*}
 Let $(M,F,d\mu)$ be a closed Finsler space and $(F_t)_{t\in [0,T]}$ be a solution to the Finsler Ricci flow with $F_0=F$.
  Assume that there exists some positive constants $K_i$, $i=1,2,3,4$, such that for all $t\in [0,T]$, $F_t$ has isotropic $S$-curvature, 
 $S=\si F_t + d\varphi$ for some time dependent functions $\si$ and $\varphi$ with $-K_3 F_t\leq d\si$ and $-K_4 F_t^2\leq Hess\ \varphi$, 
 and its Ricci curvature  satisfies
 $-K_1\leq Ric_{ij}\leq K_2$.
  If $u:[0,T]\times M\rightarrow \R$ is a positive solution of the heat equation under the Finsler Ricci flow $(F_t)$ and 
 $\be, \la \in C^1((0,T])$ are functions satisfying assumptions $(C_1)-(C_3)$
%  \begin{itemize}
%   \item[$(C_1)$] $\be (t)\in (0,1)$, $\fl\ t \in (0,T]$;
%   \item[$(C_2)$] $\displaystyle \lim_{t\rightarrow 0^+} \la(t)=0$ and $\la(t)>0$, $\fl\ t \in (0,T]$;
%   \item[$(C_3)$] $\frac{2\be'}{1-\be}-(\ln \la)'<0$ on $(0,T]$; 
%  \end{itemize} 
 then  for any $t\in (0,T]$, we have,
 \begin{equation}
\be F(\na f)^2-\pt f \leq  \frac{n}{2\be}\left((\ln \la)'-\frac{2\be'}{1-\be}\right)+\frac{n(C_1+\be')}{2\be(1-\be)}+
  \frac{n^{3/2}\sqrt{C_2}}{\be}+\sqrt{2nC_3},
 \end{equation}
where $C_1:=K_1$, $C_2:=\max\{K_1^2,K_2^2\}$ and $C_3=K_3+K_4$ and $f=\ln u$.
\end{theorem}

\begin{remark}
When  all Finsler metrics $F_t$, $t\in [0,T]$, have vanishing $S$-curvature, we have 
 \begin{equation}\label{eq:v-grad-flow}
\be F(\na f)^2-\pt f \leq  \frac{n}{2\be}\left((\ln \la)'-\frac{2\be'}{1-\be}\right)+\frac{n(C_1+\be')}{2\be(1-\be)}+
  \frac{n^{3/2}\sqrt{C_2}}{\be},
 \end{equation}
 Particulary, taking $\be=\frac{1}{\theta}$ as constant function $(\theta >1)$ and $\la(t)=t$ in  (\ref{eq:v-grad-flow}), one obtains 
 \begin{equation*}
  F(\na f)^2-\theta \pt f \leq \frac{n\theta^2}{2t}+ \frac{nC_1\theta^3}{2(\theta-1)} +n^{3/2}\theta^2\sqrt{C_2}.
 \end{equation*}
This estimate improves Lakzian result \cite[Theorem 1.1]{Lak-15}.
\end{remark}

\begin{corollary}\label{cor:flow}
 Let $(F_t)_{[0,t)}$ be a Finsler Ricci flow on a closed $n$-dimensional manifold $M$ such that 
 each Finsler metric $F_t$ has Ricci curvature satisfying 
 $0\leq Ric_{ij}\leq C$ for some positive constnate $C$ and isotropic $S$-curvature $S_t=\si F_t+d\varphi$ where 
 $\si$ and $\varphi$ satisfy assumption (\ref{eq:is-curv}).
 
 Let $b\in C^1((0,T])$ be a positive  increasing function such that $\lim_{t\rightarrow 0^+}b(t)=0$. 
 Then a positive solution $u$ of the heat equation under the Ricci flow 
 satisfies
 \begin{equation}
  F(\na f)^2-(1+b)\pt f\leq n(1+b)^2\left(\frac{b'}{b}+C\sqrt{n}+\frac{\sqrt{2nC_3}}{n(1+b)}\right),
 \end{equation}
where $f=\ln u$.

Particulary, if all Finsler metrics $F_t$ have  vanishing $S$-curvature, then we have 
\begin{equation}
 F(\na f)^2-(1+b)\pt f\leq n(1+b)^2\left(\frac{b'}{b}+C\sqrt{n}\right)
\end{equation}

\end{corollary}

\begin{proof}
 The result follows by choosing $\be=\frac{1}{1+b}$ and $\la=b$ in Theorem \ref{thm:gradient-flow*}. 
 Here $C_1=0$ and $C_2=C^2$.
\end{proof}

\begin{remark}
Some examples of functions $b$ satisfying assumptions of Corollary \ref{cor:flow} are $b(t)=\te t$ with $\te>0$ and $b(t)=\sinh(t)$.
%  Particulary, for $b(t)=\theta t$, $(\theta >0)$, we have 
% \begin{equation}
%  F(\na f)^2-(1+\theta t)\pt f \leq n(1+\theta t)^2\left(\frac{1}{t}+C\sqrt{n}\right),
% \end{equation}
% and when $b(t)=\sinh(t)$, one obtains
% \begin{equation}
%  F(\na f)^2-(1+\sinh(t))\pt f \leq n(1+\sinh(t))^2\left(\coth(t)+ C\sqrt{n}\right),
% \end{equation}
\end{remark}

\begin{proposition}
 Under hypothesis of Theorem \ref{thm:gradient-flow*}, we have, for any $x,y\in M$ and $0<s<t\leq T$, 
 \begin{equation}
  u(s,x)\leq u(t,y)\exp\left\{A(s,t)+\frac{B(s,x,t,y)}{4(t-s)}+(t-s)\sqrt{2nC_3}\right\},
 \end{equation}
where 
\begin{equation*}
 A(s,t)= (t-s) \int_0^1 \left\{\frac{n}{2\be}\left((\ln \la)'-\frac{2\be'}{1-\be}\right)+\frac{n(C_1+\be')}{2\be(1-\be)}
           +\frac{n^{3/2}\sqrt{C_2}}{\be}\right\}\ d\tau,
\end{equation*}
and 
\begin{equation*}
 B(s,x,t,y)=\inf_c \int_0^1 \frac{F_{\tilde\tau}(c(\tau),\dot c(\tau))^2}{\be(\tilde \tau)}\ d\tau,
\end{equation*}
where the infimum is taken over all the smooth paths $c:[0,1]\ra M$ satisfying $c(0)=y$  and $c(1)=x$.
\end{proposition}

\begin{proof}
Let $c:[0,1]\ra M$ be a smooth curve with $x=c(1)$ and $y=c(0)$.
 Let $f=\ln u$ and define $l(\tau)=f(\tilde \tau,c(\tau))$ for $\tau\in [0,1]$ with $\tilde \tau:= (1-\tau)t+\tau s$. Then, we have 
 \begin{eqnarray*}
  \frac{\pa l(\tau)}{\pa \tau} &=& (t-s)\left(\frac{df_{c(\tau)}(\tilde\tau,\dot c(\tau))}{t-s}-\pt f(\tilde\tau,c(\tau))\right)\\
                               &\leq& (t-s)\left(\frac{F_{\tilde \tau}(c(\tau),\na f)F_{\tilde\tau}(c(\tau),\dot c(\tau))}{t-s}
                               -\pt f(\tilde\tau,c(\tau))\right)\\
                               &\leq& (t-s)\left(\frac{1}{2\be(\tilde \tau)}\frac{F_{\tilde\tau}(c(\tau),\dot c(\tau))^2}{2(t-s)^2}
                               +\frac{\be(\tilde\tau)}{2}\times 2 F_{\tilde \tau}(c(\tau),\na f)^2-\pt f(\tilde\tau,c(\tau))\right)\\
                               &\leq&\frac{F_{\tilde\tau}(c(\tau),\dot c(\tau))^2}{4\be(\tilde \tau)(t-s)}\\
                               &+& (t-s)\left\{\frac{n}{2\be}\left((\ln \la)'-\frac{2\be'}{1-\be}\right)+\frac{n(C_1+\be')}{2\be(1-\be)}
                               +\frac{n^{3/2}\sqrt{C_2}}{\be}+\sqrt{2nC_3}\right\}
 \end{eqnarray*}
Integration of this inequality yields
\begin{eqnarray*}
 \ln \frac{u(s,x)}{u(t,y)} &=& l(1)-l(0)= \int_0^1 \frac{\pa l(\tau)}{\pa \tau}\ d\tau\\
          &\leq& (t-s)\sqrt{2nC_3}+\frac{1}{4(t-s)} \int_0^1 \frac{F_{\tilde\tau}(c(\tau),\dot c(\tau))^2}{\be(\tilde \tau)}\ d\tau\\
          &+&(t-s) \int_0^1 \left\{\frac{n}{2\be}\left((\ln \la)'-\frac{2\be'}{1-\be}\right)+\frac{n(C_1+\be')}{2\be(1-\be)}
           +\frac{n^{3/2}\sqrt{C_2}}{\be}\right\}\ d\tau
\end{eqnarray*}
Exponentiating this inequality gives immediately the required estimate.
\end{proof}

\end{document}